\documentclass[12pt]{amsart}
\hoffset= - 0.75 in

\pdfoutput=1

\usepackage{amssymb,amsmath,amsfonts,epsfig,latexsym}
\usepackage{graphicx}
\usepackage{wrapfig}
 
\newtheorem{theorem}{Theorem}[section]

\newtheorem{proposition}[theorem]{Proposition}

\newtheorem*{Richman}{Richman's Theorem}
\newtheorem*{RandomTurnHex}{Random-Turn Hex Theorem}

\theoremstyle{definition}

\def\<{\ensuremath{\langle}}
\def\>{\ensuremath{\rangle}}

\begin{document}
\title{Artificial intelligence for Bidding Hex}
\author[Payne]{Sam Payne} 
\author[Robeva]{Elina Robeva} 
\address{Department of Mathematics, Bldg 380, Stanford University, Stanford, CA 94305} \email{spayne@stanford.edu} \email{erobeva@stanford.edu} 

\thanks{Supported by Clay Mathematics Institute and a Stanford University VPUE grant.}

\begin{abstract}
We present a Monte Carlo algorithm for efficiently finding near optimal moves and bids in the game of Bidding Hex.  The algorithm is based on the recent solution of Random-Turn Hex by Peres, Schramm, Sheffield, and Wilson together with Richman's work connecting random-turn games to bidding games. \end{abstract}
\maketitle

\section{Introduction}
Hex is well-known for the simplicity of its rules and the complexity of its play.  Nash's strategy-stealing argument shows that a winning strategy for the first player exists, but finding such a strategy is intractable by current methods on large boards.  It is not known, for instance, whether the center hex is a winning first move on an odd size board.  The development of artificial intelligence for Hex is a notoriously rich and challenging problem, and has been an active area of research for over thirty years \cite{Davis75, Nishizawa76, Anshelevich00, Cazenave01, Anshelevich02, Anshelevich02b, RasmussenMaire04}, yet the best programs play only at the level of an intermediate human \cite{MelisHayward03}.  Complete analysis of Hex is essentially intractable; the problem of determining which player has a winning strategy from a given board position is PSPACE-complete \cite{Reisch81}, and the problem of determining whether a given empty hex is \emph{dead}, or irrelevant to the outcome of the game, is NP-complete \cite{BHJvR07}. Some recent research has focused on explicit solutions for small boards \cite{HBJKPvR04, HBJKPvR}, but it is unclear whether such techniques will eventually extend to the standard $11 \times 11$ board.

Bidding Hex is a variation on Hex in which, instead of alternating moves, the players bid for the right to move.   Suppose, for example, that Alice plays against Bob, and both start with 100 chips.  If Alice bids seventeen for the first move and Bob bids nineteen, then Bob gives nineteen chips to Alice and makes a move.  Now Alice has 119 chips and Bob has 81, and they bid for the second move.  The total number of chips in the game remains fixed, and the chips have no value at the end of the game---the goal is simply to win the game of Hex.  The main goal of this paper is to present an efficient algorithm for near-optimal play of Bidding Hex.  We have implemented this algorithm and it is overwhelmingly effective---undefeated against human opponents.  Code is available from the authors on request.

\section{Bidding games}

Bidding games are variations on traditional two-player games in which, instead of alternating moves, the players bid for the right to move (or for the right to determine who moves next, in situations where moving may not be desirable).  Richman developed an elegant theory of such games with real-valued bidding in the late 1980s; his theory is presented in \cite{LLPU, LLPSU}.  For recreational play, integer valued bidding is preferred.  Jay Bhat has developed Bidding Tit-Tac-Toe and Bidding Hex for online play with discrete bidding; see 
\[
\mbox{\textsf{http:/\!/apps.facebook.com/biddingttt/} \ and \ \textsf{http:/\!/apps.facebook.com/biddinghex/}}
\]
respectively.  Bidding Chess has also achieved some popularity among fans of Chess variations \cite{Beasley08, Beasley08b}.

For mathematical purposes, the coarseness of integer bidding with small numbers of chips creates additional subtleties.  For instance, the set of optimal first moves in Bidding Tic-Tac-Toe depends on the total number of chips---the center is the unique optimal first move when the total number of chips is greater than 26, but the center is not an optimal first move when the total number of chips is five \cite[Theorem~6.17]{DevelinPayne}.  

One of Richman's insights was the fundamental connection between bidding games with real-valued bidding and random-turn games in which players flip a coin to determine who moves next.  Say $G$ is a finite, loop-free game played by Alice and Bob.  Let $R(G)$ be the critical threshold, between zero and one, such that Alice has a winning strategy for $G$ with real-valued bidding if her proportion of the total bidding resources is greater than $R(G)$, and she does not have a winning strategy if her proportion is less than $R(G)$.  Let $P(G)$ be the probability that Alice wins $G$ as a random-turn game, assuming optimal play.  Intuitively, one expects that if the game favors Alice then $R(G)$ will be closer to zero and $P(G)$ will be closer to one, but the precise relation discovered by Richman remains surprising.

\begin{Richman}
Let $G$ be a finite, loop-free game.  Then
\[
R(G) = 1-P(G).
\]
\end{Richman}

\noindent Furthermore, the set of optimal moves for $G$ with real-valued bidding is the same as the set of optimal moves for random-turn play, and optimal bids may be determined as follows.  For any position $v$ in $G$, let $G_v$ denote the game played starting from $v$.  Define
\[
R^+(v) = \max \{ R(G_w) \ | \ \mbox{ Bob can move from $v$ to $w$} \}\]
and
\[
R^-(v) = \min \{ R(G_{w'}) \ | \ \mbox{ Alice can move from $v$ to $w'$} \}.
\]
The function $R$ is then determined by the relation
\[
R(G_v) = \big(R^+(v) + R^-(v)\big)/2,
\]
together with the conditions that $R(G_v)$ is equal to zero if $v$ is a winning position for Alice, and is equal to one if $v$ is a winning position for Bob.  From position $v$, an optimal bid for both players is $\delta(v) = \big|R^+(v) - R^-(v)\big| \big/ 2$,which may be thought of, roughly speaking, as the value of moving from position $v$.

Since $R^+(v)$ and $R^-(v)$ can be interpreted, through Richman's Theorem, as probabilities of winning certain random-turn games, results and insights from random-turn games can be applied directly to bidding games with real-valued bidding.  For discrete bidding, such probabilistic methods are also useful, though less directly.  For any fixed positive $\epsilon$, if Alice's proportion of the total number of chips is at least $R(G) + \epsilon$ and the total number of chips is sufficiently large with respect to $G$ and $\epsilon$, then Alice has a winning strategy in which all of her moves are optimal for real-valued bidding and hence also for random-turn play \cite[Theorem~3.10]{DevelinPayne}.  The bounds on the number of chips required are astronomical for most interesting games and grow exponentially with $1/\epsilon$ and the number of possible turns in the game, but in practice following real-valued bidding strategy is highly effective, even when the number of chips is small.

\section{Hex background}

The game of Hex never ends in a draw.  In other words, the only way a player can block the other is by completing a winning chain.  We sketch a proof of this fact following \cite{Gale79} in the form of an algorithm for determining the winner of a completed game.  This algorithm is a key ingredient in the implementation of step (1) in our artificial intelligence program for Bidding Hex, presented in Section~\ref{program}.

Suppose Alice and Bob place amber and blue hexes, respectively, until the board is full.  Color the regions northwest and southeast of the board amber, and the regions northeast and southwest of the board blue, so Alice wins if she makes a chain of amber hexes connecting the two amber regions, and Bob wins if he creates a chain of blue hexes connecting the two blue regions.

\smallskip

\begin{center}
\includegraphics{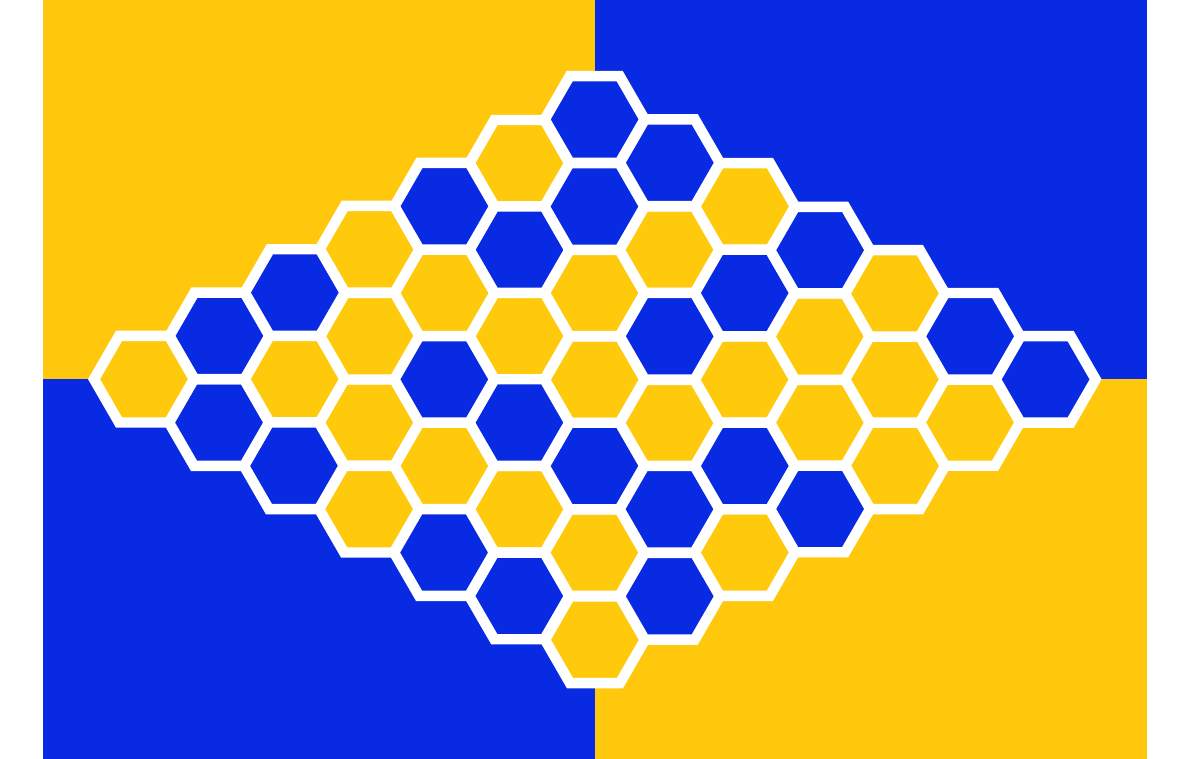}
\end{center}

\smallskip

Now the boundary between the amber and blue regions of the board is a finite union of loops, plus two paths whose four endpoints are at the north, south, east, and west corners of the board.  The directed path beginning at the west corner of the board always has amber on its left and blue on its right, so its endpoint must be at either the north or the south corner of the board.  If its endpoint is at the north corner then the blue hexes adjacent to this path on the right form a winning chain for Bob; if its endpoint is south, as in the figure below, then the amber hexes adjacent to this path on the left form a winning chain for Alice.

\smallskip

\begin{center}
\includegraphics{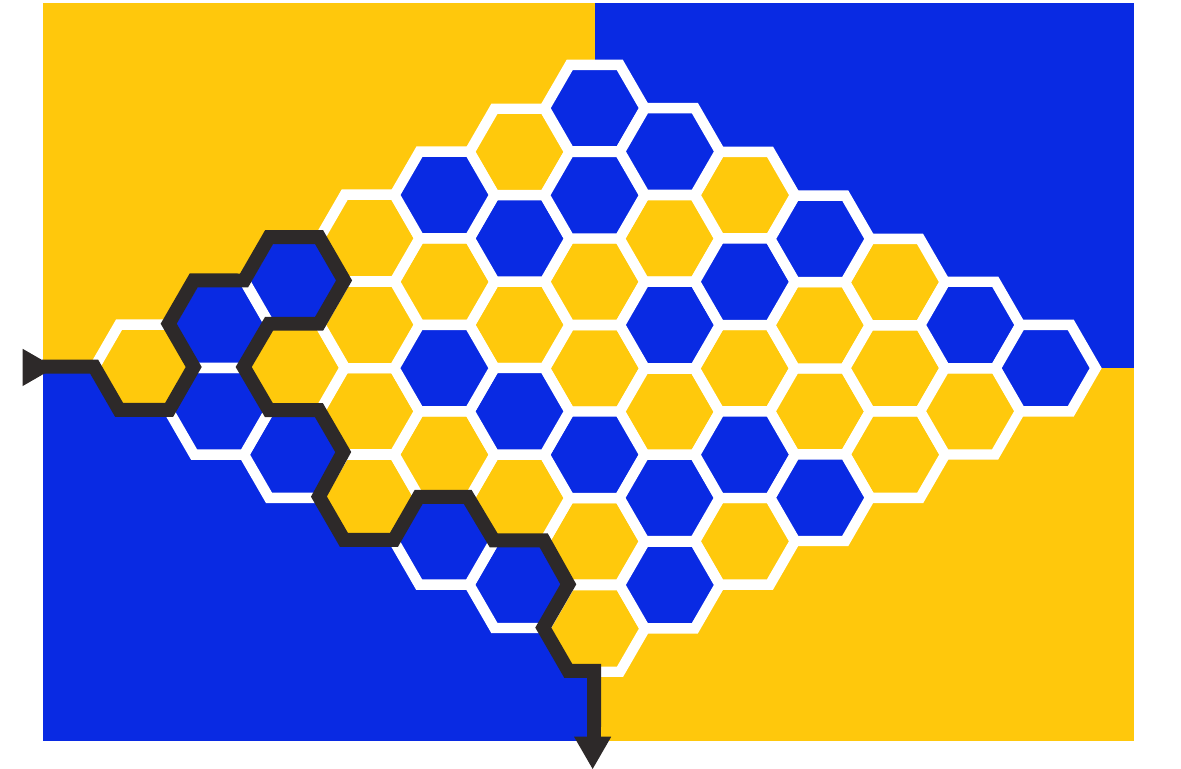}
\end{center}

\smallskip

In summary, to determine the winner of a completed game, start at the west corner and follow the boundary between amber and blue regions.  If the path leads to the north corner of the board, then Bob wins.  Otherwise, the path leads to the south corner of the board, and Alice wins.  

If the board is only partially completed, we can determine the state of the game as follows.  First, color all of the empty hexes amber and determine the winner.  If Bob is the winner, then he already has a winning chain, so the game is over.  Otherwise, color all empty hexes blue to check whether Alice already has a winning chain.  If not, then neither player has won yet, and the game continues.

\section{Random-Turn Hex}

Random-Turn Hex has recently been solved by the probabilists Peres, Schramm, Sheffield, and Wilson; their study was apparently motivated by suggestions of Werner pointing toward a conformal invariance property for Hex played on unusually shaped boards with vanishingly small hexes, based on relations with percolation theory.  One of their main results is the Random-Turn Hex Theorem below.

In a completely filled hex board, a hex is \emph{critical} if when we change the color of the token residing in it, the winner of the game changes.  In other words, a hex is critical if every winning chain contains that hex.  If the board is only partially filled in, then the \emph{probability} of an open hex being critical is the number of fillings of the remainder of the board for which the given hex is critical, divided by the total number of possible fillings, which is $2^N$, where $N$ is the number of empty hexes.

\begin{RandomTurnHex}
\cite{PSSW07}  The set of optimal moves from any position in Random-Turn Hex is the same for both players, and is exactly the set of open hexes whose probability of being critical is maximal.
\end{RandomTurnHex}

The Random-Turn Hex Theorem gives an exact solution to Random-Turn Hex.  From any position, fill in the empty hexes in every possible way and count the number of times that each hex is critical.  The hexes that are critical most often are the optimal moves.  When the number of empty hexes is large, the number of possible fillings is astronomical, so finding the optimal moves exactly is practically impossible.  Nevertheless, the Random-Turn Theorem suggests a Monte Carlo algorithm---if one fills in the board at random a large number of times, then the hexes that are critical most often in this sampling are likely to be optimal or near-optimal.  This algorithm has been implemented by David Wilson in the program Hexamania \cite{Hexamania}.  Wilson's program beats a skilled human player more than half the time, but a person can still beat the computer fairly often, by randomly winning most of the coin flips.  We have implemented a similar algorithm for Bidding Hex, finding not only near-optimal moves but also near-optimal bids, and the resulting program plays exceedingly well.

Cohen independently experimented with an analogous Monte Carlo method for playing traditional alternating move Hex, in which random hexes are filled in with alternating colors, so the end filling has equal numbers of hexes of both colors if the board size is even, and first player has exactly one extra hex in the final filling if the board size is odd \cite{Cohen04, Cohen04b}.

\section{Bidding Hex}  \label{program}
We now describe our artificial intelligence program for playing Bidding Hex.  Given a partially filled board, we must efficiently find a near optimal move and a near optimal bid.  Start with a partially filled board.  For an empty hex $H$, let $L_H$ be the probability that $H$ is filled with the losing color, when the remainder of the board is filled in at random.

\begin{proposition}
The probability that $H$ is not critical is $2 L_H$.
\end{proposition}

\begin{proof}
In any filling, if $H$ is of the losing color then it is not critical.  Furthermore, there is a natural bijection between fillings in which $H$ is of the losing color and those in which $H$ is of the winning color but not critical, given by switching the color of $H$.  Therefore, the number of fillings in which $H$ is not critical is exactly twice the number of fillings in which $H$ is of the losing color.
\end{proof}

The proposition shows that the optimal moves are those hexes $H$ such that $L_H$ is as small as possible.  An optimal bid can also be determined in terms of $L_H$, as follows.

\begin{proposition}
Let $H$ be an open hex such that $L_H$ is minimal, from position $v$.  Then an optimal bid for real-valued bidding is the proportion
\[
\delta(v) = \frac{1}{2} - L_H
\]
of the total bidding resources.
\end{proposition}

\begin{proof}
The probability that $H$ is critical is $1 - 2 L_H$, and is equal to $R^+(v) - R^-(v)$, the difference between the conditional probability that Alice wins if she gets $H$ and the probability that she wins if Bob gets $H$.  Therefore, $\delta(v) = (1 - 2L_H) /2$, as required.
\end{proof}

Here is our algorithm for finding near optimal bids and moves from a partially filled Hex board.

\begin{enumerate}
\item Fill in the remainder of the board at random a large number of times, and determine the outcome of the game each time.
\item Keep track of the number of times that each open hex is of the losing color.
\item If $H$ is the open hex that is least often of the losing color, bid the integer part of $\frac{1}{2} - L_H$ times the total number of chips in the game.
\item If this bid wins, move in hex $H$.  Otherwise, the opponent has probably overbid.  Take his chips, and let him move wherever he wants.
\end{enumerate}

\noindent A moderately fast desktop computer, by 2008 standards, can check roughly fifty thousand random fillings of an $11 \times 11$ board per second.  A few hundred thousand fillings per move suffice to consistently defeat the best human players today.

\bibliography{math}

\newcommand{\etalchar}[1]{$^{#1}$}
\providecommand{\bysame}{\leavevmode\hbox to3em{\hrulefill}\thinspace}
\providecommand{\MR}{\relax\ifhmode\unskip\space\fi MR }
\providecommand{\MRhref}[2]{%
  \href{http://www.ams.org/mathscinet-getitem?mr=#1}{#2}
}
\providecommand{\href}[2]{#2}
\begin{thebibliography}{BHJvR07}

\bibitem[Ans00]{Anshelevich00}
V.~Anshelevich, \emph{The game of {H}ex: An automatic theorem proving approach
  to game programming}, Proceedings of the Seventeenth National Conference on
  Artificial Intelligence, AAAI Press, 2000, pp.~189--194.

\bibitem[Ans02a]{Anshelevich02}
\bysame, \emph{The game of {H}ex: the hierarchical approach}, More games of no
  chance ({B}erkeley, {CA}, 2000), Math. Sci. Res. Inst. Publ., vol.~42,
  Cambridge Univ. Press, Cambridge, 2002, pp.~151--165.

\bibitem[Ans02b]{Anshelevich02b}
\bysame, \emph{A hierarchical approach to computer {H}ex}, Artificial
  Intelligence \textbf{134} (2002), 101--120.

\bibitem[Bea08a]{Beasley08}
J.~Beasley, \emph{Bidding {C}hess}, Variant Chess \textbf{7} (2008), no.~57,
  42--44.

\bibitem[Bea08b]{Beasley08b}
\bysame, \emph{Isolated pawns}, Variant Chess \textbf{7} (2008), no.~58, 71.

\bibitem[BHJvR07]{BHJvR07}
Y.~Bj{\"o}rnsson, R.~Hayward, M.~Johanson, and J.~van Rijswijck, \emph{Dead
  cell analysis in {H}ex and the {S}hannon game}, Graph theory in {P}aris,
  Trends Math., Birkh\"auser, Basel, 2007, pp.~45--59.

\bibitem[Caz01]{Cazenave01}
T.~Cazenave, \emph{Retrograde analysis of patterns versus metaprogramming},
  Computational intelligence in games, Studies in Fuzziness and Soft Computing,
  vol.~62, Physica-Verlag, 2001, pp.~57--75.

\bibitem[Coh04a]{Cohen04}
B.~Cohen, \emph{Board evaluation in hex},
  \textsf{http:/\!/www.advogato.org/person/Bram/diary/112.html}, July 2004.

\bibitem[Coh04b]{Cohen04b}
\bysame, \emph{Hex},
  \textsf{http:/\!/www.advogato.org/person/Bram/diary/116.html}, August 2004.

\bibitem[Dav76]{Davis75}
M.~Davis, \emph{On artificial machine learning: some ideas in search of a
  theory}, Internat. J. Comput. Math. \textbf{5} (1975/76), no.~4, 315--329.

\bibitem[DP08]{DevelinPayne}
M.~Develin and S.~Payne, \emph{Discrete bidding games}, preprint, 2008.

\bibitem[Gal79]{Gale79}
D.~Gale, \emph{The game of {H}ex and the {B}rouwer fixed-point theorem}, Amer.
  Math. Monthly \textbf{86} (1979), no.~10, 818--827.

\bibitem[HBJ{\etalchar{+}}04]{HBJKPvR04}
R.~Hayward, Y.~Bj{\"o}rnsson, M.~Johanson, M.~Kan, N.~Po, and J.~van Rijswijck,
  \emph{Solving {$7\times 7$} {H}ex: virtual connections and game-state
  reduction}, Advances in computer games ({G}raz, 2003), IFIP Int. Fed. Inf.
  Process., vol. 135, Kluwer Acad. Publ., Boston, MA, 2004, pp.~261--278.

\bibitem[HBJ{\etalchar{+}}05]{HBJKPvR}
\bysame, \emph{Solving {$7\times7$} {H}ex with domination, fill-in, and virtual
  connections}, Theoret. Comput. Sci. \textbf{349} (2005), no.~2, 123--139.

\bibitem[LLP{\etalchar{+}}99]{LLPSU}
A.~Lazarus, D.~Loeb, J.~Propp, W.~Stromquist, and D.~Ullman,
  \emph{Combinatorial games under auction play}, Games Econom. Behav.
  \textbf{27} (1999), no.~2, 229--264.

\bibitem[LLPU96]{LLPU}
A.~Lazarus, D.~Loeb, J.~Propp, and D.~Ullman, \emph{Richman games}, Games of no
  chance (Berkeley, CA, 1994), Math. Sci. Res. Inst. Publ., vol.~29, Cambridge
  Univ. Press, Cambridge, 1996, pp.~439--449.

\bibitem[MH03]{MelisHayward03}
G.~Melis and R.~Hayward, \emph{{SIX} wins {H}ex tournament}, ICGA \textbf{26}
  (2003), no.~4, 277--280.

\bibitem[Nis76]{Nishizawa76}
T.~Nishizawa, \emph{An improved version of the program for {H}ex}, Studies on
  puzzles by games and computers, Research Institute for Mathematical Sciences,
  Kyoto, 1976, pp.~40--61.

\bibitem[PSSW07]{PSSW07}
Y.~Peres, O.~Schramm, S.~Sheffield, and D.~Wilson, \emph{Random-turn hex and
  other selection games}, Amer. Math. Monthly \textbf{114} (2007), no.~5,
  373--387.

\bibitem[Rei81]{Reisch81}
S.~Reisch, \emph{Hex ist {PSPACE}-vollst\"andig}, Acta Inform. \textbf{15}
  (1981), no.~2, 167--191.

\bibitem[RM04]{RasmussenMaire04}
R.~Rasmussen and F.~Maire, \emph{An extension of the {H}-search algorithm for
  artificial {H}ex players}, AI 2004: {A}dvances in artifical intelligence,
  Springer, 2004, pp.~646--657.

\bibitem[Wil]{Hexamania}
D.~Wilson, \emph{Hexamania},
  \textsf{http:/\!/research.microsoft.com/$\sim$dbwilson/hex/}.

\end{thebibliography}
\bibliographystyle{amsalpha}

\end{document}